\documentclass[a4paper,preprint,12pt]{elsarticle}

\usepackage[T1]{fontenc}
\usepackage{amsmath}
\usepackage{amssymb}
\usepackage{amsthm}
\usepackage{tikz}
\usepackage{comment}
\usepackage[hidelinks]{hyperref}

\newtheorem{theorem}{Theorem}[section]
\newtheorem{proposition}[theorem]{Proposition}
\newtheorem{lemma}[theorem]{Lemma}
\newtheorem{corollary}[theorem]{Corollary}
\newtheorem{definition}[theorem]{Definition}
\theoremstyle{remark}
\newtheorem{remark}[theorem]{Remark}
\newtheorem{example}[theorem]{Example}

\newcommand{\supp}{\operatorname{supp}}
\newcommand{\cone}{\operatorname{cone}}
\newcommand{\conv}{\operatorname{conv}}
\journal{Discrete Optimization}
\biboptions{sort&compress}

\hypersetup{
  pdftitle={Common Optimal Supports and Bottleneck Hierarchies in Truncated Cones},
  pdfauthor={Walid Ben-Ameur and Alessandro Maddaloni},
  pdfkeywords={common optimal support; blocker; bottleneck hierarchy; polyhedral combinatorics; box-integrality; integer decomposition property}
}

\begin{document}

\begin{frontmatter}

\title{Common Optimal Supports and Bottleneck Hierarchies in Truncated Cones}

\author[tsp]{Walid Ben-Ameur}
\ead{walid.benameur@telecom-sudparis.eu}
\author[tsp]{Alessandro Maddaloni}
\ead{alessandro.maddaloni@telecom-sudparis.eu}

\affiliation[tsp]{organization={SAMOVAR, T\'el\'ecom SudParis, Institut Polytechnique de Paris},
                  city={Palaiseau},
                  country={France}}

\begin{abstract}
 {Let \(K\subseteq\mathbb R_+^E\) be a closed convex cone and
\(P=K\cap[0,1]^E\). For a linear objective with positive optimum, we
study the intersection of the coordinate supports of all optimal
solutions.} We prove that this common support is a transversal of a
natural clutter associated with the positive directions of the cone
and introduce a decreasing hierarchy of bottleneck families that
stabilizes after at most \(|E|-1\) levels. In the polyhedral setting,
the first level consists exactly of the sets carrying feasible
upper-bound dual multipliers, while every set carrying an optimal
multiplier belongs to the stabilized family.
{The breadth of the join-semilattice generated under union by the
zero sets of the nonzero vertices gives an objective-independent upper
bound on the universal stabilization depth;} this bound is exact for
simplicial truncations. Every possible depth occurs, and the
\(|E|-1\) bound is sharp. Membership in the second level is
co-NP-complete and fixed-parameter tractable in the size of the
candidate bottleneck. Under upper-box-integrality and the integer
decomposition property, the hierarchy collapses at level two, and its
stabilized members are exactly the carriers of optimal upper-bound
multipliers. When \(P\) is integral, the common optimal support is the
union of the supports of all optimal dual solutions. Applications
include bipartite matchings, maximum-order cycle subdigraphs, binary
circulations, balanced hypergraph matchings, interval hypergraphs, and
maximum-weight closures.
\end{abstract}


\end{frontmatter}

\section{Introduction}

{We study which coordinates are positive in every optimal solution
of a linear optimization problem over a unit-box truncation of a
nonnegative cone. Our focus is structural rather than algorithmic: we
seek descriptions of this common optimal support in terms of the
positive directions of the cone, an associated hierarchy of
bottlenecks, and dual solutions.}

Let $E$ be a finite set with $|E|\ge2$, let
$K\subseteq\mathbb R_+^E$ be a closed convex cone, and let
$P=K\cap[0,1]^E$, so that $K=\cone(P)$. Given an objective vector
$c\in\mathbb R^E$, let
$\alpha=\max\{c^\top y:y\in P\}$, and assume throughout that
$\alpha>0$. The corresponding optimal face is
$F=\{y\in P:c^\top y=\alpha\}$, and our main object of study is the common
support
$I=\bigcap_{y\in F}\supp(y)$,
where $\supp(y)=\{e\in E:y_e>0\}$. Since $F$ is compact,
$I=\bigcap_{v\in\operatorname{ext}(F)}\supp(v)$, where
$\operatorname{ext}(F)$ denotes the set of extreme points of $F$.
Indeed, if a coordinate vanishes at some point of $F$, then its minimum
over $F$ is zero and is attained at an extreme point.

{For \(0\)-\(1\) models, \(I\) is the set of variables fixed to one in
every optimum, that is, the positive part of what is called the
optimization backbone \cite{slaneywalsh2001}. In the applications of
Section~\ref{sec:applications}, \(I\) specializes to the set of vertices
covered by every maximum matching, the set of vertices contained in
every maximum-order cycle subdigraph, and the intersection of all
maximum-weight closures, among other examples. The matching case belongs
to the classical Dulmage--Mendelsohn theory \cite{dulmage1958}; the
maximum-order cycle-subdigraph problem is connected with the no-meet
matroid \cite{benameur2024nomeet}; and maximum-weight closures and the
structure of the associated minimum cuts are studied in
\cite{picard1976,picardqueyranne1980}.}

It is worth contrasting the common support $I$ with the standard primal 
optimal partition in continuous linear programming. By the Goldman--Tucker 
theorem~\cite{goldman1956theory}, the relative interior of the optimal face $F$---such 
as its analytic center~\cite{megiddo1989}---identifies the \emph{union} of the 
supports of all optimal solutions (the maximal support). In contrast, our focus 
is on characterizing the \emph{intersection} $I$ of these supports. Because $I$ 
represents the coordinates that are strictly positive in \emph{every} optimal 
solution, characterizing this set requires moving beyond standard continuous 
complementarity theory to exploit the conic and combinatorial structures 
of the underlying truncated cone.

Our approach begins with a support-transversal theorem. Associated with
the cone \(K\) is a natural clutter \(\mathcal C\) consisting of the
inclusion-minimal supports of its \(c\)-positive extreme rays. We prove
that the common support \(I\) is always a transversal of \(\mathcal C\),
which in particular implies that \(I\) is \emph{nonempty}.

We then introduce a decreasing hierarchy
\(\mathcal B_1\supseteq\mathcal B_2\supseteq\cdots\), together with its
inclusion-minimal members \(\mathcal D_1,\mathcal D_2,\ldots\). The
higher levels encode increasingly strong common-support obstructions.
The hierarchy always stabilizes by level \(|E|-1\).

In the polyhedral setting, the hierarchy has a direct dual meaning. Its
first level consists precisely of the sets carrying feasible
upper-bound multipliers, every set carrying an optimal multiplier
belongs to the stabilized family, and hence
$
\{\text{optimal dual carriers}\}
\subseteq\mathcal B_\infty
\subseteq\mathcal B_1
=\{\text{feasible dual carriers}\}.
$
The zero sets of the nonzero vertices yield a second structural
description of stabilization:
 {the breadth of the join-semilattice they generate under union gives
an objective-independent upper bound on the universal stabilization depth.}
For simplicial truncations this bound is exact. This yields simplices realizing every possible depth and, in particular,
proves that the general \(|E|-1\) bound is best possible. We also prove
that deciding membership in \(\mathcal B_2\) is co-NP-complete, while
membership is fixed-parameter tractable in the size of the candidate
bottleneck.

Our second main structural result concerns truncated cones satisfying
natural integrality properties. Under upper-box-integrality and the
integer decomposition property, the hierarchy collapses at level two,
\(\mathcal B_2=\mathcal B_3=\cdots=\mathcal B_\infty\). Moreover, these families consist exactly of the sets carrying optimal
upper-bound multipliers. If \(P\) is merely integral, the common
optimal support is the union of the supports of all optimal dual
solutions.

 {For the matching and closure models, these specializations recover
familiar structure; for the other models, they give analogous
common-support and higher-order bottleneck statements.}

The paper is organized as follows. Section~2 develops the general conic
theory and proves finite stabilization. Section~3 treats polyhedral
cones, establishes the dual-support sandwich,
 {studies the semilattice breadth of the vertex zero sets and simplicial
truncations,} and proves the complexity results.
Section~4 gives the level-two collapse and dual-core theorems under
integrality assumptions. Section~5 presents the combinatorial
applications, and Section~6 concludes with open directions.

\section{General conic case}

{As already mentioned, we assume throughout the paper that $\alpha>0$ where $\alpha:=\max\{c^\top y:y\in P\}$.}

\begin{proposition}\label{thm:main}
Let $z\in K$ satisfy $c^\top z>0$. Then
$I\cap\supp(z)\neq\emptyset.$
\end{proposition}

\begin{proof}
Suppose, for a contradiction, that $I\cap\supp(z)=\emptyset$. Then, for every $e\in\supp(z)$, there exists an optimal solution $y^e\in F$ such that $y^e_e=0$. Let
$\bar y=\frac{1}{|\supp(z)|}\sum_{e\in\supp(z)}y^e.$
Since $F$ is convex, $\bar y\in F$. Moreover, for every $e\in\supp(z)$ at least one term in the average has zero $e$-coordinate, so $\bar y_e<1$. Since $z\ge0$, there exists $\varepsilon>0$ such that $\bar y+\varepsilon z\le\mathbf1$. As $K$ is a convex cone and $\bar y,z\in K$, we have $\bar y+\varepsilon z\in K$, and hence $\bar y+\varepsilon z\in P$. But
$c^\top (\bar y+\varepsilon z)=\alpha+\varepsilon c^\top z>\alpha,$
contradicting the definition of $\alpha$.
\end{proof}

 Let $\mathcal C$ be the clutter of inclusion-minimal supports $\supp(z)$ such that $z\in K$ and $c^\top z>0$. Recall that the blocker $b(\mathcal C)$ (in the sense of Fulkerson \cite{fulkerson1970}) is the family of inclusion-minimal sets $B\subseteq E$ such that $B\cap C\neq\emptyset$ for every $C\in\mathcal C$.

\begin{corollary}\label{cor:blocker}
The set $I$ is a transversal of $\mathcal C$. Consequently, it contains a member of $b(\mathcal C)$.
\end{corollary}

\begin{proof}
By Proposition~\ref{thm:main}, $I$ intersects every member of $\mathcal C$. Hence $I$ contains an inclusion-minimal transversal of $\mathcal C$, that is, a member of $b(\mathcal C)$.
\end{proof}

The clutter $\mathcal C$ can also be described in terms of extreme rays.

\begin{proposition}\label{prop:poly}
Let $K\subseteq\mathbb R_+^E$ be a closed convex cone. Then
\[
\mathcal C=
\min_{\subseteq}\{\supp(r): r\text{ generates an extreme ray of }K,\ c^\top r>0\}.
\]
\end{proposition}

\begin{proof}
Since $K\subseteq\mathbb R_+^E$ is closed and convex, applying Minkowski's theorem and Carath\'eodory's theorem to the compact
base
$
K\cap\{x\in\mathbb R_+^E:\mathbf1^\top x=1\},
$
every vector of \(K\) is a finite conic combination of extreme rays of
\(K\). \\
Let $C\in\mathcal C$, and choose $z\in K$ such that $\supp(z)=C$ and $c^\top z>0$. Write
$z=\sum_i\lambda_i r^i,$
where $\lambda_i>0$ and each $r^i$ generates an extreme ray of $K$. Since all vectors are nonnegative, $\supp(r^i)\subseteq C$ for every $i$. Since $c^\top z>0$, at least one ray $r^i$ satisfies $c^\top r^i>0$. By the inclusion-minimality of $C$, we get $\supp(r^i)=C$. \\
Conversely, let $r$ be a $c$-positive extreme ray whose support is inclusion-minimal among the supports of $c$-positive extreme rays. Since $r\in K$ and $c^\top r>0$, the set $\supp(r)$ contains some member $C\in\mathcal C$. By the first part, $C$ is the support of a $c$-positive extreme ray. The minimality of $\supp(r)$ gives $C=\supp(r)$.
\end{proof}

Since any nonzero element of $K$ can be scaled to an element of $P$ with the same support, Proposition~\ref{thm:main} is equivalent to the same statement for every $z\in P$ with $c^\top z>0$. 
The following theorem is a higher-order version.

\begin{theorem}\label{th:general}
Let $z^1,\ldots,z^\ell\in P$ satisfy
$\sum_{i=1}^\ell c^\top z^i>\alpha(\ell-1).$
Then
$
\left(\bigcap_{i=1}^\ell\supp(z^i)\right)\cap I\neq\emptyset.
$
\end{theorem}

\begin{proof}
The case $\ell=1$ is Proposition~\ref{thm:main}. Let $\ell\ge2$. If $\bigcap_{i=1}^\ell\supp(z^i)=\emptyset$, then every coordinate is zero in at least one of the vectors $z^i$. Hence
$\bar z=\frac{1}{\ell-1}\sum_{i=1}^\ell z^i$
belongs to $P$, while $c^\top \bar z>\alpha$, a contradiction. Thus $S:=\bigcap_{i=1}^\ell\supp(z^i)$ is nonempty.

Assume now, for a contradiction, that $S\cap I=\emptyset$. For every $e\in S$, choose $y^e\in F$ such that $y^e_e=0$. Let $r=|S|$ and
$w=\sum_{i=1}^\ell z^i+\sum_{e\in S}y^e.$
For any coordinate $e'\in S$, the vector $y^{e'}$ is zero at $e'$. For any coordinate $e'\notin S$, at least one of the vectors $z^i$ is zero at $e'$. Hence every coordinate is positive in at most $\ell+r-1$ of the vectors appearing in the sum defining $w$. Therefore
$\frac{1}{\ell+r-1}w\in P.$
Moreover,
$c^\top w=\sum_{i=1}^\ell c^\top z^i+r\alpha>\alpha(\ell-1)+r\alpha=\alpha(\ell+r-1),$
contradicting the optimality of $\alpha$.
\end{proof}

Equivalently, Theorem~\ref{th:general} says that, for feasible vectors $z^1,\ldots,z^\ell$,
\[
\left(\bigcap_{i=1}^\ell\supp(z^i)\right)\cap I=\emptyset
\quad\Longrightarrow\quad
\sum_{i=1}^\ell c^\top z^i\le\alpha(\ell-1).
\]
This can be viewed as a weighted pigeonhole principle for supports: if $\ell$ feasible points have total value larger than $(\ell-1)$ times the optimum, then they share a coordinate that is present in every optimal solution. Equivalently, if $\delta(z)=\alpha-c^\top z$ denotes the deficit of $z$, then
$\sum_{i=1}^\ell\delta(z^i)<\alpha \quad\Longrightarrow\quad I\cap\bigcap_{i=1}^\ell\supp(z^i)\neq\emptyset.$ {This leads naturally to the study of $\ell$-bottleneck sets, defined formally below:} 
\begin{definition}[$\ell$-bottleneck property]
A set $J\subseteq E$ satisfies the $\ell$-bottleneck property, or simply is an $\ell$-bottleneck set, if for all $z^1,\ldots,z^\ell\in P$,
\[
\left(\bigcap_{i=1}^\ell\supp(z^i)\right)\cap J=\emptyset
\quad\Longrightarrow\quad
\sum_{i=1}^\ell c^\top z^i\le\alpha(\ell-1).
\]
\end{definition}

Let $\mathcal B_\ell$ denote the family of all $\ell$-bottleneck sets. If $J\in\mathcal B_\ell$ and $J\subseteq J'\subseteq E$, then $J'\in\mathcal B_\ell$.

\begin{proposition}
\label{prop:monotone}
For every $\ell\ge1$,
$\mathcal B_{\ell+1}\subseteq\mathcal B_\ell.$
\end{proposition}

\begin{proof}
Let $J\in\mathcal B_{\ell+1}$, and let $z^1,\ldots,z^\ell\in P$ satisfy $\sum_{i=1}^\ell c^\top z^i>\alpha(\ell-1)$. Let $z^{\ell+1}$ be any optimal solution. Then
$\sum_{i=1}^{\ell+1}c^\top z^i>\alpha\ell.$
Since $J\in\mathcal B_{\ell+1}$, the set $J$ intersects $\bigcap_{i=1}^{\ell+1}\supp(z^i)$, and therefore it also intersects $\bigcap_{i=1}^{\ell}\supp(z^i)$. Hence $J\in\mathcal B_\ell$.
\end{proof}

We say that $J$ satisfies the $\infty$-bottleneck property if it satisfies the $\ell$-bottleneck property for every integer $\ell\ge1$. Let $\mathcal B_\infty$ be the corresponding family. By Theorem~\ref{th:general}, $I\in\mathcal B_\infty$.

\begin{theorem}
\label{1pro:J}
A nonempty set $J\subseteq E$ satisfies the $\infty$-bottleneck property if and only if it satisfies the $|J|$-bottleneck property. Consequently,
$\mathcal B_\infty=\mathcal B_{|E|-1}.$
\end{theorem}

\begin{proof}
Let $J=\{e_1,\ldots,e_{|J|}\}$ and suppose that $J\in\mathcal B_{|J|}$. Assume, for a contradiction, that $J\notin\mathcal B_\ell$ for some $\ell>|J|$. Then there exist $z^1,\ldots,z^\ell\in P$ such that
$\sum_{i=1}^\ell c^\top z^i>\alpha(\ell-1)$
and
$\left(\bigcap_{i=1}^\ell\supp(z^i)\right)\cap J=\emptyset.$
For each $e_j\in J$, choose an index $i_j$ such that $e_j\notin\supp(z^{i_j})$. Let $Q\subseteq\{1,\ldots,\ell\}$ be the set of chosen indices. Then $|Q|\le |J|$ and
$\left(\bigcap_{i\in Q}\supp(z^i)\right)\cap J=\emptyset.$
Since $c^\top z^i\le\alpha$ for every $i$, we have
$\sum_{i\in Q}c^\top z^i \ge \sum_{i=1}^\ell c^\top z^i-(\ell-|Q|)\alpha > \alpha(|Q|-1).$
This contradicts $J\in\mathcal B_{|Q|}$, which follows from $J\in\mathcal B_{|J|}$ and Proposition~\ref{prop:monotone}. Thus $J\in\mathcal B_\infty$. The reverse implication is immediate.

It remains to prove the displayed equality. Since $\mathcal B_\infty\subseteq\mathcal B_{|E|-1}$ is immediate, let $J\in\mathcal B_{|E|-1}$. If $J=E$, then $J\in\mathcal B_\infty$. Otherwise, $|J|\le |E|-1$, and Proposition~\ref{prop:monotone} gives $J\in\mathcal B_{|J|}$. By the first part, $J\in\mathcal B_\infty$.
\end{proof}

Define the inclusion-minimal bottleneck clutters
$
\mathcal D_\ell=\min_{\subseteq}\{J\in\mathcal B_\ell\}$.
Then $\mathcal D_1=b(\mathcal C)$. Moreover, since $\mathcal B_{\ell+1}\subseteq\mathcal B_\ell$, every member of $\mathcal D_{\ell+1}$ contains at least one member of $\mathcal D_\ell$. By Theorem~\ref{1pro:J}, the hierarchy $\mathcal D_\ell$ stabilizes no later than $\ell=|E|-1$. Observe that the blocker of $\mathcal D_\ell$
is given by 
\[
b(\mathcal D_\ell)=
\min_{\subseteq}
\left\{
\bigcap_{i=1}^\ell\supp(z^i):
 z^i\in P,\ \sum_{i=1}^\ell c^\top z^i>\alpha(\ell-1)
\right\}.
\]

 {Indeed, let $\mathcal R = \min_{\subseteq}
\left\{
\bigcap_{i=1}^\ell\supp(z^i):
 z^i\in P,\sum_{i=1}^\ell c^\top z^i>\alpha(\ell-1)
\right\}$. 
  By definition, $\mathcal B_\ell$ is the set of transversals of $\mathcal R $, and $\mathcal D_\ell$ is the set of minimal elements of $\mathcal B_\ell$, implying that $\mathcal D_\ell = b(\mathcal R)$.  Consequently, $b(\mathcal D_\ell)= b (b (\mathcal R)) = \mathcal R$ using \cite{edmonds1970bottleneck}. }

{ Next theorem states that each member of $\mathcal D_\infty$ is a subset of $I$.}
\begin{theorem}\label{pro:J}
$\bigcup\limits_{J\in\mathcal D_\infty}J\subseteq I.$  
\end{theorem}
\begin{proof}
Let $J \in \mathcal D_\infty=\mathcal D_{|E|-1}$. Assume that $J\setminus I\neq\emptyset$, and let $e\in J\setminus I$. By the minimality of $J$, the set $J\setminus\{e\}$ does not satisfy the $\infty$-bottleneck property. Thus there exist $z^1,\ldots,z^\ell\in P$ such that
$\sum_{i=1}^\ell c^\top z^i>\alpha(\ell-1)$
and
$\left(\bigcap_{i=1}^\ell\supp(z^i)\right)\cap (J\setminus\{e\})=\emptyset.$
Since $e\notin I$, there exists $y\in F$ with $y_e=0$. Let $z^{\ell+1}=y$. Then
$\sum_{i=1}^{\ell+1}c^\top z^i>\alpha\ell.$
Moreover, $\bigcap_{i=1}^{\ell+1}\supp(z^i)$ does not meet $J\setminus\{e\}$ and also does not contain $e$. Hence it does not meet $J$, contradicting $J\in\mathcal B_\infty$.
\end{proof}

To test whether $J\subseteq E$ belongs to $\mathcal B_\ell$, we maximize the total value of $\ell$ feasible vectors under the condition that, for each $v\in J$, at least one of the $\ell$ vectors has zero $v$-coordinate. This gives the mixed-integer formulation
\begin{align}
V_\ell(J)=\max\quad &\sum_{i=1}^\ell c^\top z^i \nonumber\\
\text{s.t.}\quad & z^i\in K\cap[0,1]^E && i=1,\ldots,\ell,\nonumber\\
& z^i_v\le y^i_v && i=1,\ldots,\ell,\ v\in J, \label{eq:IP}\\
& y^i_v\in\{0,1\} && i=1,\ldots,\ell,\ v\in J,\nonumber\\
& \sum_{i=1}^\ell y^i_v\le \ell-1 && v\in J.\nonumber
\end{align}
By definition,
$J\in\mathcal B_\ell\quad\Longleftrightarrow\quad V_\ell(J)\le\alpha(\ell-1).$

\section{Polyhedral cones: dual carriers,
\texorpdfstring{ {semilattice breadth}}{semilattice breadth}, and complexity}
\label{sec:polyhedral}

Throughout this section, \(K\subseteq\mathbb R_+^E\) is a rational
polyhedral cone. Let
\(K^*=\{s\in\mathbb R^E:s^\top x\geq0\text{ for every }x\in K\}\)
be its dual cone. By linear-programming duality,
{
\(\alpha=\max\{c^\top x:x\in K, x\leq\mathbf1\}
=\min\{\mathbf1^\top u:u\geq0,\ u-c\in K^*\}\)}. The vector \(u\) is the multiplier vector associated with the upper
bounds \(x_e\leq1\).

\subsection{Feasible and optimal dual carriers}

Define
\(\mathfrak F=\{J\subseteq E:\exists u\geq0,\ \supp(u)\subseteq J,
\ u-c\in K^*\}\) and
\(\mathfrak O=\{J\subseteq E:\exists u\geq0,\ \supp(u)\subseteq J,
\ u-c\in K^*,\ \mathbf1^\top u=\alpha\}\). Thus \(\mathfrak F\) is the family of sets carrying a feasible
upper-bound multiplier, while \(\mathfrak O\) is the family of sets
carrying an optimal one.

\begin{theorem}[Dual-support sandwich]
\label{prop:dual-sandwich}
For every rational polyhedral cone \(K\subseteq\mathbb R_+^E\),
\(\mathfrak O\subseteq\mathcal B_\infty
\subseteq\mathcal B_1=\mathfrak F\). Consequently, \(\mathcal D_1\) is the family of inclusion-minimal
supports of feasible upper-bound multiplier vectors.
\end{theorem}

\begin{proof}
Let \(J\in\mathfrak O\), witnessed by \(u\), and let
\(z^1,\ldots,z^\ell\in P\) have common support disjoint from \(J\).
For every \(e\in\supp(u)\), at least one of the entries
\(z^1_e,\ldots,z^\ell_e\) is zero, and hence
\(\sum_{i=1}^{\ell}z^i_e\leq\ell-1\). Since
\(u-c\in K^*\), \(u\geq0\), and \(\mathbf1^\top u=\alpha\),
$
\sum_{i=1}^{\ell}c^\top z^i
\leq
u^\top\sum_{i=1}^{\ell}z^i
\leq
(\ell-1)\mathbf1^\top u
=
(\ell-1)\alpha.
$
Thus \(J\in\mathcal B_\infty\). The inclusion
\(\mathcal B_\infty\subseteq\mathcal B_1\) is immediate.

It remains to prove \(\mathcal B_1=\mathfrak F\). Let
\(L_J=\{x\in\mathbb R^E:x_e=0\text{ for every }e\in J\}\). By scaling,
\(J\in\mathcal B_1\Longleftrightarrow
c^\top x\leq0\text{ for every }x\in K\cap L_J\). Since \(K\) is polyhedral,
\((K\cap L_J)^*=K^*+L_J^\perp\). Therefore
\(-c=s+w\) for some \(s\in K^*\) and some vector \(w\) supported in
\(J\). Write \(w=w^+-w^-\), where $w^+$ and $w^-$ are nonnegative vectors. As
\(K\subseteq\mathbb R_+^E\), every nonnegative vector belongs to
\(K^*\). Hence, with \(u=w^-\), we have
\(u\geq0\), \(\supp(u)\subseteq J\), and \(u-c=s+w^+\in K^*\). Thus \(J\in\mathfrak F\). Conversely, if \(J\in\mathfrak F\), witnessed
by \(u\), and \(x\in K\cap L_J\), then
\(c^\top x\leq u^\top x=0\); hence \(J\in\mathcal B_1\).
\end{proof}


\subsection{\texorpdfstring{ {Semilattice breadth}}{Semilattice breadth}
and simplicial truncations}

Let
$
\operatorname{ext}(P)\setminus\{0\}=\{r^1,\ldots,r^m\}.
$
For \(j\in[m]\), let \(Z_j=\{e\in E:r^j_e=0\}\). We write
\(\mathcal Z(P)=(Z_1,\ldots,Z_m)\).
 {We denote by \(\operatorname{br}_{\cup}(\mathcal Z(P))\) the breadth
of the finite join-semilattice generated by \(\mathcal Z(P)\) under set
union \cite{ditor1984}. Equivalently,
\(\operatorname{br}_{\cup}(\mathcal Z(P))\) is the smallest integer
\(q\geq1\) such that, for every \(S\subseteq[m]\), there exists
\(T\subseteq S\), with \(|T|\leq q\), satisfying}
\(
 {\bigcup_{j\in T}Z_j=\bigcup_{j\in S}Z_j.}
\)
We write \(\mathcal B_\ell(c)\) when we wish to emphasize the
dependence of the bottleneck hierarchy on the objective, and define
the \emph{universal stabilization depth} of \(P\) by
\[
s(P)=\min\{q\geq1:\mathcal B_q(c)=\mathcal B_\infty(c)
\text{ for every }c\text{ with }\alpha(c)>0\}.
\]

\begin{theorem}[ {Semilattice-breadth bound}]
\label{prop:semilattice-breadth-bound}
For every polyhedral truncation $P=K\cap[0,1]^E$,
 {$s(P)\leq \operatorname{br}_{\cup}(\mathcal Z(P))$.}
Consequently,
 {$s(P)\leq\min\{|E|-1,\operatorname{br}_{\cup}(\mathcal Z(P))\}$.}
\end{theorem}

\begin{proof}
 {Let \(q=\operatorname{br}_{\cup}(\mathcal Z(P))\), fix an objective \(c\)}
with positive optimum \(\alpha\), and suppose that
\(J\notin\mathcal B_\ell(c)\) for some \(\ell\). Choose a violating
family \(z^1,\ldots,z^\ell\). Each \(c^\top z^i\) is positive, since
otherwise the remaining \(\ell-1\) terms would sum to at most
\((\ell-1)\alpha\). {Therefore, for each $e \in J$, there exists an index $i \in [\ell]$ such that $z^i_e = 0$.}

Write each \(z^i\) as a convex combination of the vertices of \(P\).
For every \(i\), some nonzero vertex \(r^{j_i}\) occurring with
positive coefficient satisfies \(c^\top r^{j_i}\geq c^\top z^i\).
Since all points of \(P\) are nonnegative, every zero coordinate of
\(z^i\) is also zero in \(r^{j_i}\). Hence
\(J\subseteq\bigcup_{i=1}^{\ell}Z_{j_i}\) and
\(\sum_{i=1}^{\ell}c^\top r^{j_i}>(\ell-1)\alpha\).
Put \(d_j=\alpha-c^\top r^j\geq0\). The sum of the deficits of the
selected occurrences is smaller than \(\alpha\).
 {By the definition of semilattice breadth, at most \(q\) of the
selected zero sets have the same union.} The sum of their deficits is
still smaller than \(\alpha\).
After adding copies of an optimal vertex if necessary, we obtain
\(q\) points whose zero sets cover \(J\) and whose total value is
larger than \((q-1)\alpha\). Thus \(J\notin\mathcal B_q(c)\), proving
\(\mathcal B_q(c)=\mathcal B_\infty(c)\). The second inequality follows
also from Theorem~\ref{1pro:J}.
\end{proof}

\begin{remark}
 {The semilattice-breadth bound need not be exact for a nonsimplicial}
truncation. For instance, if \(K=\mathbb R_+^2\), then
\(P=[0,1]^2\). The zero sets of its nonzero vertices are
\(\{1\}\), \(\{2\}\), and \(\emptyset\), so
 {\(\operatorname{br}_{\cup}(\mathcal Z(P))=2\).} On the other hand,
Theorem~\ref{1pro:J} gives
\(\mathcal B_1(c)=\mathcal B_\infty(c)\) for every objective with
positive optimum, since \(|E|-1=1\). Hence \(s(P)=1<2\).
\end{remark}

{We will now focus on the simplex case. Recall that
\(
P=\conv\{0,r^1,\ldots,r^m\}
\)
is a simplex if the $r^j$ are linearly independent. This requires that
$|E|\geq m$.}

\begin{theorem}[Universal depth of a simplicial truncation]
\label{th:simplicial-breadth}
Let $K\subseteq\mathbb R_+^E$ be a polyhedral cone and suppose that
$
P=K\cap[0,1]^E=\operatorname{conv}\{0,r^1,\ldots,r^m\}$
is a simplex. Let
$
Z_j:=\{e\in E:r^j_e=0\}, j\in[m]$,
and let $\mathcal Z(P)=\{Z_1,\ldots,Z_m\}$. Then
$
 {s(P)=\operatorname{br}_{\cup}(\mathcal Z(P))}$.
\end{theorem}

\begin{proof}
 {Theorem~\ref{prop:semilattice-breadth-bound} gives}
$
 {s(P)\leq\operatorname{br}_{\cup}(\mathcal Z(P))}$.
We prove the reverse inequality. Let
 {$q<\operatorname{br}_{\cup}(\mathcal Z(P))$. By the definition of
semilattice breadth, there exists $S\subseteq[m]$ such that}
$
J:=\bigcup_{j\in S}Z_j
$
cannot be represented as the union of at most $q$ members of
$\{Z_j:j\in S\}$.

Since $P$ is a simplex, the nonzero vertices
$r^1,\ldots,r^m$ are linearly independent. Hence there exists
$c\in\mathbb R^E$ such that
\[
c^\top r^j=
\begin{cases}
\alpha,&j\in S,\\
0,&j\notin S,
\end{cases}
\]
for some $\alpha>0$. Thus the optimum value of $c$ over $P$ is
$\alpha$.

We first show that $J\notin\mathcal B_\infty(c)$. The vertices
$r^j$, $j\in S$, have zero sets whose union is $J$, while
$
\sum_{j\in S}c^\top r^j=|S|\alpha>(|S|-1)\alpha$.
Thus they violate the bottleneck condition at level $|S|$, so
$J\notin\mathcal B_{|S|}(c)$, implying that 
$
J\notin\mathcal B_\infty(c).
$

We next show that $J\in\mathcal B_q(c)$. Suppose, to the
contrary, that $J\notin\mathcal B_q(c)$. Then there exist
$x^1,\ldots,x^q\in P$ whose common support is disjoint from $J$ and
such that
$
\sum_{i=1}^q c^\top x^i>(q-1)\alpha$.
In particular, since $c^\top x^i\leq\alpha$ for every $i$, we have
$c^\top x^i>0$ for every $i$. Write each $x^i$ as a convex
combination of $0,r^1,\ldots,r^m$. Since
$c^\top r^j=0$ for $j\notin S$, for each $i$ there is some
$j_i\in S$ occurring with positive coefficient. Hence
$
 Z(x^i):=\{e \in E : x^i_e = 0\} \subseteq Z_{j_i}$.
Moreover, the common support of the $x^i$ is disjoint from $J$, so
$
J\subseteq\bigcup_{i=1}^q Z(x^i)
\subseteq\bigcup_{i=1}^q Z_{j_i}
\subseteq\bigcup_{j\in S}Z_j=J.
$
Thus
$
J=\bigcup_{i=1}^q Z_{j_i}$,
contradicting the choice of $S$. Therefore $J\in\mathcal B_q(c)$.

We have therefore found $c$ and $J$ such that
$
J\in\mathcal B_q(c)\setminus\mathcal B_\infty(c)$ implying that $s(P)>q$. Since this holds for every
 {$q<\operatorname{br}_{\cup}(\mathcal Z(P))$, we obtain}
$
 {s(P)\geq\operatorname{br}_{\cup}(\mathcal Z(P)).}
$
 {Together with Theorem~\ref{prop:semilattice-breadth-bound}, this gives}
$
 {s(P)=\operatorname{br}_{\cup}(\mathcal Z(P))}$.
\end{proof}

\begin{proposition}[All universal stabilization depths occur]
\label{th:all-depths}
For every pair of integers \(1\leq q\leq m\), there exists a rational
simplicial cone \(K_{m,q}\subseteq\mathbb R_+^{m+1}\) such that
\(P_{m,q}=K_{m,q}\cap[0,1]^{m+1}\) is a simplex with \(m\) nonzero vertices and
universal stabilization depth exactly \(q\).
\end{proposition}

\begin{proof}
Choose positive integers \(m_1+\cdots+m_q=m\), and let
\(X_1,\ldots,X_q\) be pairwise disjoint sets, where
\(X_i=\{x_{i,1},\ldots,x_{i,m_i}\}\). For \(i\in[q]\) and
\(t\in[m_i]\), define \(Z_{i,t}=\{x_{i,1},\ldots,x_{i,t}\}\). Within each block the sets form a chain. Hence any union of members of
this family is represented by retaining at most the largest selected
 {member of each chain, so its semilattice breadth is at most \(q\).} On the
other hand, the \(q\) terminal sets
\(Z_{i,m_i}=X_i\), \(i\in[q]\), form an irredundant subfamily (that is, their union  cannot be expressed as the union of less than $q$ sets).
 {Therefore the semilattice breadth equals \(q\).}

Let \(E=\{0\}\mathbin{\dot\cup}X_1\mathbin{\dot\cup}\cdots
\mathbin{\dot\cup}X_q\), and define
\(r^{i,t}=\mathbf1-\chi^{Z_{i,t}}\) and
\(K_{m,q}=\cone\{r^{i,t}:i\in[q],\ t\in[m_i]\}\). Observe that the $r^{i,t}$ are linearly independent. 
If \(x=\sum_{i,t}\lambda_{i,t}r^{i,t}\in K_{m,q}\), then
$x_0=\sum_{i,t}\lambda_{i,t}$. Consequently, if $x$ is in  $K_{m,q}\cap[0,1]^E $,  then it can be written as $x = (1-x_0) 0 + \sum_{i,t}\lambda_{i,t}r^{i,t}$ proving that $x \in \conv\bigl(\{0\}\cup\{r^{i,t}:i\in[q],\ t\in[m_i]\}\bigr)$. Therefore
$K_{m,q}\cap[0,1]^E=
\conv\bigl(\{0\}\cup\{r^{i,t}:i\in[q],\ t\in[m_i]\}\bigr)$.
Theorem~\ref{th:simplicial-breadth} shows that its universal depth is $q$.
\end{proof}

\begin{corollary}[Sharpness of the general stabilization bound]
\label{cor:sharpness}
For every integer \(m\geq2\), there exist a simplicial rational cone
\(K_m\subseteq\mathbb R_+^{m+1}\) and a unit objective \(c\) such
that, for every \(\ell\geq1\),
$
\mathcal B_\ell
=
\{J\subseteq E:0\in J\}
\cup
\{J\subseteq[m]:|J|\geq\ell+1\},
$
where \(E=\{0,1,\ldots,m\}\). Consequently,
\(\mathcal B_1\supsetneq\mathcal B_2\supsetneq\cdots
\supsetneq\mathcal B_m=\mathcal B_\infty\). Thus the bound \(|E|-1\) in Theorem~\ref{1pro:J} is best possible,
even when \(P\) is a  simplex.
\end{corollary}

\begin{proof}
Take \(r^j=\mathbf1-e_j\) for \(j\in[m]\),
\(K_m=\cone\{r^1,\ldots,r^m\}\), and \(c=e_0\). This is the case \(q=m\) of Proposition~\ref{th:all-depths}, and
\(P_m=\conv\{0,r^1,\ldots,r^m\}\). Every nonzero point of \(K_m\)
has positive \(0\)-coordinate. Hence, if \(0\in J\), any family whose
common support avoids \(J\) contains the zero point, and
\(J\in\mathcal B_\ell\).

Now let \(J\subseteq[m]\). If
\(z=\sum_{j=1}^m\lambda_jr^j\neq0\) and
\(t=\sum_j\lambda_j\), then \(z_e=t-\lambda_e\). Thus a nonzero point
can vanish in at most one coordinate of \([m]\). If
\(|J|\geq\ell+1\), every \(\ell\)-tuple whose common support avoids
\(J\) contains the zero point, so its total value is at most
\(\ell-1\). Conversely, if \(|J|\leq\ell\), use the vertices
\(r^e\), \(e\in J\), and complete the tuple with copies of any one of
them; for \(J=\emptyset\), use arbitrary optimal vertices. The common
support avoids \(J\) and the total value is \(\ell\). This proves the
formula.
\end{proof}

\subsection{Complexity of bottleneck recognition}

We now study the complexity of $\mathcal B_\ell$-membership. 
While membership in $\mathcal B_1$ can be tested by one linear program, membership in $\mathcal B_2$ is more delicate. Indeed,
\(
J\notin\mathcal B_1
\text{ if and only if }
\max\{c^\top z:z\in P,\ z_e=0\text{ for every }e\in J\}>0.
\)
We next show that this tractability does not extend to the second layer.

\begin{theorem}[Co-NP-completeness of second-level recognition]\label{thm:B2-conp}
Given a rational polyhedral cone $K\subseteq\mathbb R^E_+$ described by rational linear inequalities, a rational objective vector $c\in\mathbb Q^E$, and a subset $J\subseteq E$, where $P=K\cap[0,1]^E$ and $\alpha=\max\{c^\top z:z\in P\}>0$, deciding whether $J\in\mathcal B_2$ is co-NP-complete. The result remains true when $c$ is a unit vector, $\alpha=1$, and $K$ is described by a polynomial number of homogeneous linear equations and nonnegativity constraints.
\end{theorem}

\begin{proof}
A certificate for $J\notin\mathcal B_2$ is given by two subsets $S_1,S_2\subseteq E$ such that $S_1\cap S_2\cap J=\emptyset$. For $i\in\{1,2\}$, consider the linear program
$\beta_i=\max\{c^\top z:z\in P,\ z_e=0\text{ for every }e\notin S_i\}.$
If $\beta_1+\beta_2>\alpha$, then there exist $z^1,z^2\in P$ such that $c^\top z^1+c^\top z^2>\alpha$ and $(\supp(z^1)\cap\supp(z^2))\cap J=\emptyset$, so $J\notin\mathcal B_2$. Conversely, the supports of any violating pair provide such sets. Hence non-membership belongs to NP, and membership belongs to co-NP.

We prove co-NP-hardness by reduction from the directed two vertex-disjoint paths problem, which is NP-complete \cite{fortune-hopcroft-wyllie}. An instance consists of a digraph $D=(V,A)$ and four pairwise distinct vertices $s_1,t_1,s_2,t_2$; the question is whether $D$ contains vertex-disjoint directed paths from $s_1$ to $t_1$ and from $s_2$ to $t_2$. 
We can safely assume that each terminal pair is individually linked.

Construct the node-split digraph $\widetilde D=(\widetilde V,\widetilde A)$ by replacing every vertex $v\in V$ with two vertices $v^-$ and $v^+$ joined by the resource arc $h_v=(v^-,v^+)$, and every original arc $(u,v)\in A$ with the arc $(u^+,v^-)$. Let $B$ be the node-arc incidence matrix of $\widetilde D$, with a $+1$ at the tail and a $-1$ at the head of each arc. Introduce nonnegative variables $f^1,f^2\in\mathbb R_+^{\widetilde A}$, $x\in\mathbb R_+^V$, and $\lambda_1,\lambda_2,\tau\in\mathbb R_+$. The coordinate set of the constructed instance is
$E'=\{f^k_a:k\in\{1,2\},\ a\in\widetilde A\}\cup\{x_v:v\in V\}\cup\{\lambda_1,\lambda_2,\tau\}.$
Define the cone 
\[
K = \left\{
\begin{aligned}
&(f^1,f^2,x,\lambda_1,\lambda_2,\tau)\ge 0 : \\
&Bf^k=\lambda_k(\chi_{s_k^-}-\chi_{t_k^+}), &&k=1,2, \\
&x_v=f^1_{h_v}+f^2_{h_v}, &&v\in V, \\
&\tau=\lambda_1+\lambda_2
\end{aligned}
\right\}.
\]

Let $P=K\cap[0,1]^{E'}$, let $c$ be the unit vector corresponding to the coordinate $\tau$, and put
$J=\{\lambda_1,\lambda_2\}\cup\{x_v:v\in V\}.$
Since $\tau\le1$ on $P$, one has $\alpha\le1$. Conversely, a unit flow along any $s_k^-$-$t_k^+$ path, with $\lambda_k=\tau=1$, gives a point of $P$. Hence $\alpha=1$.

Suppose first that the original instance has vertex-disjoint paths $Q_1,Q_2$. Let $z^1\in P$ send one unit of commodity $1$ along the path corresponding to $Q_1$, with $(\lambda_1)^1=\tau^1=1$ and $(\lambda_2)^1=0$, and define $z^2$ symmetrically from $Q_2$. Then $c^\top z^1+c^\top z^2=2>\alpha$. Moreover, neither $\lambda_1$ nor $\lambda_2$ is positive in both points, and vertex-disjointness implies that no coordinate $x_v$ is positive in both points. Hence $J\notin\mathcal B_2$.

Conversely, suppose that $J\notin\mathcal B_2$. Then there exist $z^1,z^2\in P$ such that $\tau^1+\tau^2>1$ and $(\supp(z^1)\cap\supp(z^2))\cap J=\emptyset$. Both $\tau^1$ and $\tau^2$ are positive. Since $\lambda_1,\lambda_2\in J$, each coordinate $\lambda_k$ is positive in at most one of the two points. One point cannot have both $\lambda_1$ and $\lambda_2$ positive, for then both would vanish in the other point and its value of $\tau$ would be zero. After possibly exchanging $z^1$ and $z^2$, we have $(\lambda_1)^1>0$, $(\lambda_2)^1=0$, $(\lambda_1)^2=0$, and $(\lambda_2)^2>0$.

The positive support of the commodity-$1$ flow in $z^1$ contains an $s_1^-$-$t_1^+$ path, and the positive support of the commodity-$2$ flow in $z^2$ contains an $s_2^-$-$t_2^+$ path. If these two paths used the same original vertex $v$, both would traverse $h_v$, and hence $x_v^1>0$ and $x_v^2>0$, contradicting $x_v\in J$. Contracting resource arcs gives the required vertex-disjoint paths in $D$. Therefore non-membership in $\mathcal B_2$ is NP-hard, and membership is co-NP-complete.
\end{proof}

\begin{proposition}[Fixed-parameter tractability]\label{prop:fpt-small-bottlenecks}
Assume that linear optimization over $P$, with additional prescribed zero coordinates, is polynomial-time solvable. For fixed $\ell$, membership of $J$ in $\mathcal B_\ell$ can be decided by solving at most $\ell^{|J|}$ linear programs; hence it is fixed-parameter tractable with parameter $|J|$ for every fixed $\ell$. Membership in $\mathcal B_\infty$ can be decided by solving at most $|J|^{|J|}$ linear programs, and is therefore FPT parameterized by $|J|$.
\end{proposition}
\begin{proof}
The set $J$ fails to belong to $\mathcal B_\ell$ iff there are $z^1,\ldots,z^\ell\in P$ with $\sum_i c^\top z^i>\alpha(\ell-1)$ and such that each $e\in J$ is zero in at least one of the vectors. Equivalently, there is a map $\sigma:J\to\{1,\ldots,\ell\}$ with $z^{\sigma(e)}_e=0$ for every $e\in J$. For each $\sigma$ we solve the LP
\[
\max\left\{\sum_{i=1}^\ell c^\top z^i:z^i\in P\ \forall i,\ z^{\sigma(e)}_e=0\ \forall e\in J\right\}.
\]
One LP has value larger than $\alpha(\ell-1)$ iff $J\notin\mathcal B_\ell$. There are $\ell^{|J|}$ maps. For $\mathcal B_\infty$, Theorem~\ref{1pro:J} reduces the test to $\ell=|J|$ whenever $J\neq\emptyset$; the empty set is never an $\infty$-bottleneck set because $\alpha>0$.
\end{proof}

\section{Integral truncated cones}\label{sec:integer}

Throughout this section, let $K\subseteq\mathbb R_+^E$ be a rational polyhedral cone and let $P=K\cap[0,1]^E$. We say that $K$ is \emph{upper-box-integer} if $K\cap[0,b]$ is integral for every integral vector $b\ge0$. A stronger standard notion is \emph{box-integrality}: $K$ is box-integer if $K\cap[p,q]$ is integral for every pair of integral vectors $p\le q$. Box-integrality implies upper-box-integrality, but the aggregation result below only needs the latter. We also assume that $P$ has the \emph{integer decomposition property} (IDP): for every integer $\ell\ge1$ and every integer point $x\in\ell P$, there exist integer points $x^1,\ldots,x^\ell\in P$ such that $x=x^1+\cdots+x^\ell$.

For $\ell\ge1$ and $J\subseteq E$, define
\[
b^{\ell,J}_e=
\begin{cases}
\ell-1,& e\in J,\\
\ell,& e\notin J,
\end{cases}
\qquad
P_\ell(J)=K\cap[0,b^{\ell,J}].
\]

\begin{lemma}\label{lem:box-idp}
Assume that $K$ is upper-box-integer and that $P$ has the IDP. Then $P_\ell(J)$ is integral. Moreover, every integer point $x\in P_\ell(J)$ admits a decomposition
$x=x^1+\cdots+x^\ell, \quad x^i\in P\cap\{0,1\}^E,$
such that for every $e\in J$, at least one of the vectors $x^1,\ldots,x^\ell$ satisfies $x^i_e=0$. Consequently,
$V_\ell(J)=\max\{c^\top x:x\in P_\ell(J)\}$, where $V_\ell(J)$ is defined in \eqref{eq:IP}.
\end{lemma}

\begin{proof}
Since $K$ is upper-box-integer and $b^{\ell,J}$ is integral, $P_\ell(J)$ is integral. Let $x\in P_\ell(J)$ be an integer point. Since $P_\ell(J)\subseteq K\cap[0,\ell]^E=\ell P$, the IDP of $P$ gives
$x=x^1+\cdots+x^\ell, \qquad x^i\in P\cap\mathbb Z^E.$
Since $P\subseteq[0,1]^E$, each $x^i$ is binary. If $e\in J$, then $x_e\le\ell-1$, and therefore at least one of the binary entries $x^1_e,\ldots,x^\ell_e$ is zero.

It remains to prove the equality for $V_\ell(J)$. Given a feasible solution $z^1,\ldots,z^\ell$ of \eqref{eq:IP}, the vector $x=\sum_{i=1}^\ell z^i$ belongs to $K$, satisfies $x_e\le\ell$ for all $e$, and satisfies $x_e\le\ell-1$ for $e\in J$. Hence $x\in P_\ell(J)$, and the objectives agree. This proves
$V_\ell(J)\le\max\{c^\top x:x\in P_\ell(J)\}.$
Conversely, since $P_\ell(J)$ is integral, there is an optimal integer point $x\in P_\ell(J)$. Decompose it as above. For every $e\in J$, at least one summand has zero $e$-coordinate, so the vectors $x^1,\ldots,x^\ell$ can be completed with suitable binary variables $y^i_e$ to a feasible solution of \eqref{eq:IP}. The objective value is $c^\top x$. Thus the reverse inequality holds.
\end{proof}

\begin{theorem}[Level-two collapse]\label{th:stagn}
Let $K\subseteq\mathbb R_+^E$ be an upper-box-integer rational polyhedral cone, let $P=K\cap[0,1]^E$, and assume that $P$ has the IDP. Let $c\in\mathbb Q^E$, let $\alpha=\max\{c^\top x:x\in P\}$, and assume that $\alpha>0$. Then, for every $\ell\ge2$ and every $J\subseteq E$,
\[
\begin{aligned}
J\in\mathcal B_\ell
\quad\Longleftrightarrow\quad
&\text{there exists an optimal solution }u\text{ of the dual problem} \\
&\min\{\mathbf1^Tu:u\ge0,\ u-c\in K^*\}
\text{ with }\supp(u)\subseteq J.
\end{aligned}
\]
Consequently,
$\mathcal B_2=\mathcal B_3=\cdots=\mathcal B_\infty$ and
$\mathcal D_2=\mathcal D_3=\cdots=\mathcal D_\infty$, and
$\mathcal D_2$ is the family of inclusion-minimal supports of optimal
upper-bound multiplier vectors.
\end{theorem}

\begin{proof}
By Lemma~\ref{lem:box-idp},
$V_\ell(J)=\max\{c^\top x:x\in P_\ell(J)\}$. By strong duality, the dual
of the aggregated problem over $P_\ell(J)$ is
$V_\ell(J)=\min\{(b^{\ell,J})^Tu:u\ge0,\ u-c\in K^*\}$.
Since
$b^{\ell,J}=(\ell-1)\mathbf1+\chi_{E\setminus J}$,
this is
$V_\ell(J)=\min\{(\ell-1)\mathbf1^Tu+u(E\setminus J):
u\ge0,\ u-c\in K^*\}$.
Let $x^*$ be a base optimal primal solution. Since $K$ is a cone,
$(\ell-1)x^*\in K$, and clearly
$(\ell-1)x^*\le b^{\ell,J}$. Thus
$(\ell-1)x^*\in P_\ell(J)$, so
$V_\ell(J)\ge(\ell-1)\alpha$.

Suppose first that there exists a base-dual optimal solution $u$ with
$\supp(u)\subseteq J$. Then $\mathbf1^Tu=\alpha$ and
$u(E\setminus J)=0$, so the aggregated dual gives
$V_\ell(J)\le(\ell-1)\alpha$. Hence
$V_\ell(J)=(\ell-1)\alpha$, and $J\in\mathcal B_\ell$.

Conversely, suppose $J\in\mathcal B_\ell$. Then
$V_\ell(J)\le(\ell-1)\alpha$, and therefore equality holds. Let $u$
be an optimal solution of the aggregated dual. Since $u$ is feasible
for the base dual, $\mathbf1^Tu\ge\alpha$, while
$u(E\setminus J)\ge0$. The equality
$(\ell-1)\mathbf1^Tu+u(E\setminus J)=(\ell-1)\alpha$
forces $\mathbf1^Tu=\alpha$ and $u(E\setminus J)=0$. Thus $u$ is
base-dual optimal and $\supp(u)\subseteq J$.

The characterization is independent of $\ell\ge2$, which proves the
collapse of the families $\mathcal B_\ell$ and of their
inclusion-minimal members. The description of $\mathcal D_2$ follows
by taking inclusion-minimal supports.
\end{proof}

\begin{corollary}\label{cor:stagn-algorithmic}
Under the assumptions of Theorem~\ref{th:stagn}, if $K$ is
well-described and given by a polynomial-time separation oracle, then
for every $\ell\ge2$ membership in $\mathcal B_\ell$ can be decided in
polynomial time.
\end{corollary}

\begin{proof}
By Lemma~\ref{lem:box-idp}, $V_\ell(J)$ is the optimum of
$\max\{c^\top x:x\in P_\ell(J)\}$. Adding the explicit bounds
$0\le x\le b^{\ell,J}$ to a separation oracle for $K$ gives a
polynomial-time separation oracle for $P_\ell(J)$. Linear optimization
over $P_\ell(J)$ is therefore polynomial-time solvable by the ellipsoid
method \cite{gls1993geometric}. Comparing the optimum with
$\alpha(\ell-1)$ decides whether $J\in\mathcal B_\ell$.
\end{proof}

Combining Theorem~\ref{th:stagn} with
Theorem~\ref{prop:dual-sandwich} gives
\(\mathfrak O=\mathcal B_2=\mathcal B_3=\cdots=\mathcal B_\infty\). Thus, under upper-box-integrality and the integer decomposition
property, level one consists of the carriers of feasible dual
certificates, whereas every level from two onward consists of the
carriers of optimal dual certificates.

{Let us now give a description  of $I$ in terms of  optimal dual solution supports. Observe that we are here only requiring the integrality of $P$.  }

\begin{theorem}[Dual description of the optimal core]\label{th:dual-core}
Let $K\subseteq\mathbb R_+^E$ be a rational polyhedral cone and let $P=K\cap[0,1]^E$. Assume that $P$ is integral, that is,
$P=\conv(P\cap\{0,1\}^E).$
Let $\mathcal U$ be the set of optimal solutions of
$\min\{\mathbf1^Tu:u\ge0,\ u-c\in K^*\}.$
Then
$I=\bigcup_{u\in\mathcal U}\supp(u).$
\end{theorem}

\begin{proof}
Let $u\in\mathcal U$ and let $e\in\supp(u)$. For any optimal primal solution $x\in F$, strong duality implies that 
$\mathbf1^Tu-c^\top x=u^T(\mathbf1-x)+(u-c)^Tx=0.$
Both terms on the right are nonnegative. Since $u_e>0$, it follows that $x_e=1$. Thus $e\in I$, proving
$\bigcup_{u\in\mathcal U}\supp(u)\subseteq I.$ \\
Conversely, let $e\in I$. Since $P$ is integral, every extreme point of the optimal face is binary. Hence every optimal point satisfies $x_e=1$. Choose a strictly complementary optimal primal--dual pair \cite{goldman1956theory}. The upper-bound constraint $x_e\le1$ is active at its primal component, so strict complementarity gives $u_e>0$ for the corresponding upper-bound multiplier. Hence $e\in\supp(u)$, proving the reverse inclusion.
\end{proof}

\begin{remark}\label{rem:dual-core-single-support}
The equality in Theorem~\ref{th:dual-core} is a union over all optimal dual multipliers. Equivalently, there exists an optimal dual vector $u^\circ$ such that $\supp(u^\circ)=I$. Indeed, for every $e\in I$, Theorem~\ref{th:dual-core} gives an optimal dual vector $u^e$ with $u^e_e>0$. Averaging these vectors over $e\in I$ gives an optimal dual vector that is positive on every coordinate of $I$. The first inclusion in the proof of Theorem~\ref{th:dual-core} shows that no optimal dual multiplier can be positive outside $I$. Thus the average has support exactly $I$. Notice that such a vector need not have inclusion-minimal support and need not be integral, even when integral optimal multipliers exist.
\end{remark}

We recall that a rational system \(Ax\leq b\) is totally dual integral (TDI) if, for
every integral objective vector for which the primal optimum is finite,
the associated dual has an integral optimal solution. The system is
box-TDI if the augmented system
$Ax\leq b,\quad p\leq x\leq q$
is TDI for every pair of rational vectors \(p\leq q\). A rational
polyhedral cone is box-integer if and only if it admits a box-TDI
description \cite{chervet-grappe-robert}.

\begin{remark}
Theorem~\ref{th:stagn} only requires upper-box-integrality. If, in addition, $K$ is box-integer, then by the box-integer/box-TDI equivalence recalled above it admits a box-TDI description. Hence, when $c$ is integral, the upper-bound multiplier vector $u$ appearing in Theorem~\ref{th:stagn} may be chosen integral. For rational $c$, one may scale $c$ to an integral vector; the corresponding optimal multiplier is scaled by the same factor, so its support is unchanged. This integrality assertion is existential and concerns a box-TDI representation of $K$; it does not say that every optimal multiplier, or the multipliers associated with an arbitrary description of $K$, are integral.
\label{re:box}
\end{remark}

\section{Combinatorial applications}
\label{sec:applications}

The following examples illustrate how the conic framework recovers structural properties in standard combinatorial settings. In each case we specify the ground set, the cone, the integral points of its truncation, the positive-support clutter, and the dual interpretation of the stabilized bottlenecks. In each application we use a subscripted notation for the corresponding common optimal support: for instance $I_M$ for matchings, $I_D$ for cycle subdigraphs, $I_{\rm circ}$ for binary circulations, $I_{\mathcal H}$ for balanced hypergraphs, and $I_{\rm cl}$ for closures.

\subsection{Bipartite matchings}

Let $G=(U,W;E_G)$ be a bipartite graph with edge set $E_G$, and let $M_G\in\{0,1\}^{(U\cup W)\times E_G}$ be its vertex-edge incidence matrix. We work in the ground set $U\cup W$ and define
$K_M=\cone\{\chi_e:e\in E_G\}=\{x\in\mathbb R_+^{U\cup W}:\exists y\in\mathbb R_+^{E_G},\ x=M_Gy\},$
where $\chi_e$ is the incidence vector of the endpoints of $e$. Let $P_M=K_M\cap[0,1]^{U\cup W}$ and take $c=\mathbf1$.

 Since $G$ is bipartite, the incidence matrix $M_G$ is totally unimodular \cite{schrijver1998theory} (TU). Hence $K_M$ is box-integer: for integral vectors $p \le q $, the polytope $\{y\ge0:p \le M_Gy\le q\}$ is integral, and its projection onto the relevant variables yields the integral set  $K_M\cap[p,q]$. The polytope $P_M$ has the IDP. Indeed, an integer point of $kP_M$ is the degree vector of a bipartite multigraph of maximum degree at most $k$, and K\"onig's line-colouring theorem decomposes its edge multiset into $k$ matchings.

The nonzero integral points of $P_M$ are precisely the covered-vertex vectors of matchings. Thus $\alpha=\max\{\mathbf1^Tx:x\in P_M\}=2\nu(G)$, where $\nu(G)$ is the maximum matching size. For this cone, let $I_M$ denote the set of vertices covered by every maximum matching.

The extreme rays of $K_M$ are the edge vectors $\chi_e$. Therefore $\mathcal C$ is the family of edge endpoint sets, and $\mathcal{D}_1 = b(\mathcal C)$ is the family of inclusion-minimal vertex covers. Corollary~\ref{cor:blocker} gives that the vertices covered by every maximum matching form a vertex cover.

The dual cone is
$K_M^*=\{\pi \in\mathbb R^{U\cup W}:\pi_u+\pi_w\ge0\text{ for every }uw\in E_G\}.$
The base dual is
\[
\min\left\{\sum_{v\in U\cup W}\pi_v:\pi_v\ge0,\ \pi_u+\pi_w\ge2\text{ for every }uw\in E_G\right\}.
\]
After the substitution $\pi=2p$, this is twice the vertex-cover LP.
Hence Theorem~\ref{th:stagn} yields, for every $\ell\ge2$,
$
J\in\mathcal B_\ell$
 if and only if $J$ contains a minimum vertex cover of $G$.
Consequently,
\(
\mathcal D_2=\mathcal D_3=\cdots=\mathcal D_\infty
=\{\text{minimum vertex covers of }G\}.
\)

Theorem~\ref{th:general} also yields the following support-intersection property. If $\ell$ matchings have total cardinality exceeding $(\ell-1)\nu(G)$, then there exists a vertex of $I_M$ that is saturated by all of them. Since every minimum vertex cover belongs to $\mathcal{D}_{\infty}$, it contains a vertex that is saturated by all $\ell$ matchings.

\begin{proposition}
For a bipartite graph, the set of vertices covered by every maximum matching is the union of all minimum vertex covers.
\end{proposition}

\begin{proof}
The inclusion from the union of all minimum vertex covers into the set of vertices covered by every maximum matching follows {from Theorem \ref{pro:J}.}  Conversely, let $v$ be covered by every maximum matching. Then $\nu(G-v)=\nu(G)-1$. By K\"onig's theorem, $G-v$ has a vertex cover $C$ of size $\nu(G)-1$. Then $C\cup\{v\}$ is a vertex cover of $G$ of size $\nu(G)$, so it is a minimum vertex cover containing $v$.
\end{proof}

\subsection{Maximum-order cycle subdigraphs}

Let $D=(V,A)$ be a directed graph and let $B$ be its node-arc incidence matrix. Define
\[
K_D=\left\{x\in\mathbb R_+^V:\exists y\in\mathbb R_+^A,\ By=0,\ x_v=\sum_{a\in\delta^+(v)}y_a\ \forall v\in V\right\},
\]
where $\delta^+(v)$ denotes the set of arcs leaving $v$. We take $c=\mathbf1$ and $P_D=K_D\cap[0,1]^V$. Integral points of $P_D$ are exactly vertex sets of vertex-disjoint directed cycle subdigraphs. For this cone, let $I_D$ denote the set of vertices contained in every maximum-order cycle subdigraph.

\begin{proposition}
$K_D$ is box-integer and $P_D$ satisfies the integer decomposition property.
\end{proposition}
\begin{proof}
Let $\widetilde{D} = (\widetilde{V}, \widetilde{A})$ be the split digraph obtained by splitting each vertex $v \in V$ into an in-node $v^-$ and an out-node $v^+$ connected by a resource arc $(v^-, v^+)$ with flow $x_v$, mapping each original arc $(u,v) \in A$ to $(u^+, v^-)$ with flow $y_{uv}$. The joint system defining $K_D$ under arbitrary box bounds $p, q \in \mathbb{Z}^V_+$ is governed by the node-arc incidence matrix $M = [M_y \mid M_x]$ of $\widetilde{D}$:
\[
\begin{bmatrix} M_y & M_x \\ -M_y & -M_x \\ 0 & I \\ 0 & -I \end{bmatrix} \begin{pmatrix} y \\ x \end{pmatrix} \le \begin{pmatrix} 0 \\ 0 \\ q \\ -p \end{pmatrix}, \quad y \ge 0
\]
Because $M$ is a node-arc incidence matrix, it is TU. Stacking rows of the identity matrix preserves total unimodularity, rendering the entire constraint matrix TU. 
Since the right-hand side is integral, the joint polytope is integral by the Hoffman--Kruskal theorem. Its projection onto the $x$-coordinates, $K_D \cap [p, q]$, inherits this integrality, proving that $K_D$ is box-integer.

To establish the Integer Decomposition Property  of $P_D = K_D \cap \{x \in \mathbb{R}^V : 0 \le x \le \mathbf{1}\}$, let $\bar{x} \in kP_D \cap \mathbb{Z}^V$, implying $0 \le \bar{x} \le k\mathbf{1}$. The fiber polytope $P(y) = \{y \in \mathbb{R}_+^A : M_y y = -M_x \bar{x}\}$ is non-empty and governed by a TU submatrix $M_y$ with an integral right-hand side. Thus, its extreme points are integral, meaning $\bar{x}$ lifts to a fully integral circulation $(\bar{y}, \bar{x}) \in \mathbb{Z}_+^A \times \mathbb{Z}^V$. 

This integral vector belongs to the $k$-dilation of the bounded base joint polytope:
\[
P_{y,x} = \left\{(y,x) \in \mathbb{R}_+^A \times \mathbb{R}^V : M_y y + M_x x = 0, \ 0 \le x \le \mathbf{1}\right\}.
\]

Since the constraint matrix of $P_{y,x}$ is TU, the Baum--Trotter decomposition theorem \cite{baum1978integer} guarantees that $(\bar{y}, \bar{x})$ can be written as the sum of $k$ integral points of $P_{y,x}$:
\(
\begin{pmatrix} \bar{y} \\ \bar{x} \end{pmatrix} = \sum_{i=1}^k \begin{pmatrix} y^i \\ x^i \end{pmatrix}
\)
where $(y^i, x^i) \in P_{y,x} \cap (\mathbb{Z}_+^A \times \mathbb{Z}^V)$ for each $i$. Projecting this sum onto the $x$-coordinates yields $\bar{x} = \sum_{i=1}^k x^i$. Because each $x^i \in K_D$ and $0 \le x^i \le \mathbf{1}$, we have $x^i \in P_D \cap \{0,1\}^V$, which completes the proof.
\end{proof}

Therefore, the cone $K_D$ satisfies the hypotheses of Theorem~\ref{th:stagn}. Observe that  $\alpha$ is the maximum number of vertices covered by such a subdigraph. Moreover, inclusion-minimal positive supports in $K_D$ are precisely the inclusion-minimal vertex sets of directed cycles. Thus $\mathcal D_1 = b(\mathcal C)$ is the family of minimal feedback vertex sets.

\begin{corollary}
The vertices contained in every maximum-order cycle subdigraph form a feedback vertex set of $D$.
\end{corollary}

\begin{proof}
By Corollary~\ref{cor:blocker}, the common support of all optimal solutions intersects every inclusion-minimal directed-cycle vertex set. Hence it intersects every directed cycle of $D$.
\end{proof}

The dual-cone formulation gives a more refined description. Since
$
K_D=\cone\{\chi_{V(C)}:C\text{ is a directed cycle of }D\},
$
we have
\(
K_D^*=\{s\in\mathbb R^V:s(V(C))\ge0
\text{ for every directed cycle }C\}.
\)
The base dual is
\[
\min\left\{\sum_{v\in V}u_v:u_v\ge0,\ u(V(C))\ge |V(C)|\text{ for every directed cycle }C\right\}.
\]
The support of every feasible solution $u$ of this dual is a feedback vertex set, because each directed cycle must contain a vertex with positive $u$-value. Conversely, any feedback vertex set supports a feasible dual solution after assigning sufficiently large weights to its vertices. Thus the dual can be viewed as a weighted, cycle-length version of the feedback-vertex-set covering problem. Therefore, for every $\ell\ge2$, $J\in\mathcal B_\ell$ if and only if $J$ contains the support of an optimal cycle-length-cover vector $u$, and $\mathcal D_2$ is the family of inclusion-minimal supports of such optimal vectors. Moreover, Theorem~\ref{th:dual-core} gives the explicit core identity
\[
I_D=\bigcup\{\supp(u):u\text{ is optimal for the cycle-length-cover dual}\}.
\]
By Remark~\ref{rem:dual-core-single-support}, there is also an optimal dual vector $u^\circ$ with $\supp(u^\circ)=I_D$, although this vector need not be inclusion-minimal.

{The value $\alpha$ has two further interpretations. First, it is the dimension of the no-meet matroid associated with $D$, equivalently the maximum order of a collection of vertex-disjoint directed cycles \cite{benameur2024nomeet}. Second, in the helicopter-cops/invisible-slow-robber game, a directed version of the hunters-and-rabbit game, the same number is the minimum number of capture attempts: $\alpha=d(N(D))=ca(D)$ \cite{benameur2024meeting}. By box-integrality of $K_D$, the dual has an integral optimum $u$ (Remark \ref{re:box}). This can also be deduced from the results of \cite{benameur2024meeting} by considering an optimal strategy for the cops (or hunters) and defining $u_v$ as the total number of optimal capture attempts assigned to vertex $v$. Consequently, $u$ is integral and dual optimal.
}

\begin{example}\label{ex:cycle-subdigraph}
For illustration, let $D=(V,A)$ be the digraph of
Figure~\ref{fig:intersecting_cycles} having $8$ vertices and $10$ arcs, where $V=\{1,\ldots,8\}$ and
\(
A=\{(1,2), (2,3), (3,4), (4,8), (8,1), (1,5), (5,6),  (6,7), (7,8),  (3,1)\}.
\)

\begin{figure}[htbp]
\centering
\begin{tikzpicture}[
    vertex/.style={circle, draw, minimum size=8mm, font=\small\bfseries, fill=gray!10},
    arc/.style={->, >=stealth, thick, shorten >=1pt, shorten <=1pt}
]
    \node[vertex] (v1) at (0,0.5) {1};
    \node[vertex] (v2) at (-2,1.5) {2};
    \node[vertex] (v3) at (-4,0) {3};
    \node[vertex] (v4) at (-2,-1.5) {4};
    \node[vertex] (v5) at (2,1.5) {5};
    \node[vertex] (v6) at (4,0) {6};
    \node[vertex] (v7) at (2,-1.5) {7};
    \node[vertex] (v8) at (0,-1.5) {8};

    \draw[arc] (v1) to[bend right=15] (v2);
    \draw[arc] (v2) to[bend right=15] (v3);
    \draw[arc] (v3) to[bend right=15] (v4);
    \draw[arc] (v4) to[bend right=15] (v8);
    \draw[arc] (v8) to[bend right=15] (v1);
    \draw[arc] (v3) -- (v1);
    \draw[arc] (v1) to[bend left=15] (v5);
    \draw[arc] (v5) to[bend left=15] (v6);
    \draw[arc] (v6) to[bend left=15] (v7);
    \draw[arc] (v7) to[bend left=15] (v8);
\end{tikzpicture}
\caption{A digraph with two maximum-order cycle subdigraphs intersecting at vertices $1$ and $8$.}
\label{fig:intersecting_cycles}
\end{figure}
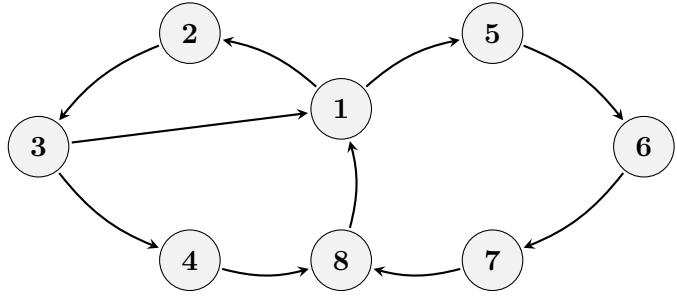

The directed cycles have vertex supports
\(
S_1=\{1,2,3\},\ 
S_2=\{1,2,3,4,8\},\
S_3=\{1,5,6,7,8\}.
\)
After retaining inclusion-minimal supports, one gets
\(
\mathcal C=\{\{1,2,3\},\{1,5,6,7,8\}\}.
\)
The blocker is
\[
\begin{aligned}
b(\mathcal C)=\{&\{1\},\{2,5\},\{2,6\},\{2,7\},\{2,8\},
                 \{3,5\},\{3,6\},\{3,7\},\{3,8\}\}.
\end{aligned}
\]

Maximizing $\mathbf1^Tx$ over $P_D$ gives $\alpha=5$. There are two
optimal binary vertices:
$
v^1=(1,1,1,1,0,0,0,1)^\top$, $   
v^2=(1,0,0,0,1,1,1,1)^\top.
$
Thus
$
I_D=\supp(v^1)\cap\supp(v^2)=\{1,8\}.
$
As predicted by Corollary~\ref{cor:blocker}, the set $I_D$ contains a member of $b(\mathcal C)$, namely $\{1\}$.

Let us compute the second bottleneck layer explicitly. The cycle-length-cover dual is to minimize $\sum_{v\in V}u_v$, subject to $u\ge0$,
$u_1+u_2+u_3\ge3$, $u_1+u_2+u_3+u_4+u_8\ge5$ and $u_1+u_5+u_6+u_7+u_8\ge5$.
Since the primal optimum is $\alpha=5$, every optimal dual has total weight $5$. The third constraint, together with optimality, forces $u_2=u_3=u_4=0$. The first two constraints then give $u_1\ge3$ and $u_1+u_8=5$, while all other variables outside $\{1,8\}$ are zero. Thus the optimal dual supports are $\{1\}$ and $\{1,8\}$, and the unique inclusion-minimal one is $\{1\}$. Hence $\mathcal D_2=\mathcal D_3=\cdots=\mathcal D_\infty=\{\{1\}\}$. In particular, the union of the minimal stabilized bottlenecks can be a proper subset of the optimal core: here $\bigcup_{J\in\mathcal D_2}J=\{1\}\subsetneq I_D=\{1,8\}$.
\end{example}

\subsection{Maximum binary circulations}

Let $D=(V,A)$ be a directed graph and let $B$ be its node-arc incidence matrix. Consider the pure circulation cone
$K_{\rm circ}=\{y\in\mathbb R_+^A:By=0\},$
and the truncated polytope
$P_{\rm circ}=K_{\rm circ}\cap[0,1]^A.$
We take $c=\mathbf1$, so the objective maximizes the number of selected arcs. For this cone, let $I_{\rm circ}$ denote the set of arcs contained in every maximum-cardinality binary circulation.

Since $B$ is a node-arc incidence matrix, the system
$By=0,\qquad 0\le y\le\mathbf1$
is totally unimodular. Hence $K_{\rm circ}$ is box-integer and $P_{\rm circ}$ is integral. The IDP of $P_{\rm circ}$ follows from the Baum--Trotter decomposition theorem for totally unimodular systems \cite{baum1978integer}. Its integral points are precisely the binary circulations of $D$, that is, arc sets $R\subseteq A$ for which every vertex has the same indegree and outdegree in $R$. Equivalently, every nonzero binary circulation is an arc-disjoint union of directed cycles.

The cone $K_{\rm circ}$ is generated by the incidence vectors of directed cycles:
$K_{\rm circ}=\cone\{\chi_{A(C)}:C\text{ is a directed cycle of }D\}.$
Therefore the positive inclusion-minimal supports of $K_{\rm circ}$ are precisely the arc sets of directed cycles. Hence $\mathcal C$ is the directed-cycle clutter on arcs, and $b(\mathcal C)$ is the family of minimal feedback arc sets.

\begin{corollary}
The arcs contained in every maximum-cardinality binary circulation form a feedback arc set.
\end{corollary}

\begin{proof}
By Corollary~\ref{cor:blocker}, the common support of all maximum-cardinality binary circulations intersects every directed cycle. This is exactly the definition of a feedback arc set.
\end{proof}

Theorem~\ref{th:general} also gives a higher-order intersection statement. If $R_1,\ldots,R_\ell$ are binary circulations and
$|R_1|+\cdots+|R_\ell|>(\ell-1)\alpha,$
where $\alpha$ is the maximum size of a binary circulation, then $R_1,\ldots,R_\ell$ share an arc that belongs to every maximum-cardinality binary circulation.

The dual description is again explicit. Since the dual cone is defined using the following set of constraints: 
$K_{\rm circ}^*=\{s\in\mathbb R^A:s(A(C))\ge0\text{ for every directed cycle }C\},$
the base dual is
\[
\min\left\{\sum_{a\in A}u_a:u_a\ge0,\ u(A(C))\ge |A(C)|\text{ for every directed cycle }C\right\}.
\]
Thus Theorem~\ref{th:stagn} yields
\(
\mathcal D_2=\cdots=\mathcal D_\infty
=\min_{\subseteq}\{\supp(u):u\text{ is dual-optimal}\}.
\)
Every member of $\mathcal D_2$ is a feedback arc set. The dual-core theorem gives
\[
I_{\rm circ}
=\bigcup\{\supp(u):u\text{ is optimal for the
cycle-length feedback-arc cover}\}.
\]
Thus an arc belongs to every maximum-cardinality binary
circulation if and only if it receives positive price in some optimal
cycle-length feedback-arc cover.

The connection with strong maximum circulations is particularly direct
in the unit-capacity case. Following Atkinson et
al.~\cite{atkinson2023strong}, let
\[
A^\star=
\{a\in A:\text{some maximum circulation }y\in P_{\rm circ}
\text{ satisfies }y_a<1\}.
\]
They prove that there exists a strong maximum circulation whose residual
arc set is exactly \(A^\star\), the union of the residual arc sets of all
maximum circulations. Since \(P_{\rm circ}\) is integral,
\(
I_{\rm circ}=A\setminus A^\star.
\)
Thus \(I_{\rm circ}\) is precisely the set of arcs saturated by every
maximum unit-capacity circulation.

After eliminating the node-potential variables from the standard
circulation dual, its upper-bound multipliers are exactly the feasible
vectors of the cycle-length feedback-arc-cover dual above. Ordinary
complementary slackness gives
\(\supp(u)\subseteq A\setminus A^\star\) for every optimal multiplier,
while the strong complementary-slackness result of Atkinson et
al.~yields an optimal multiplier \(u^\circ\) satisfying
\(\supp(u^\circ)=A\setminus A^\star\). Consequently,
\(
A\setminus A^\star
=I_{\rm circ}
=\bigcup\bigl\{\supp(u):u\text{ is optimal for the cycle-length feedback-arc cover}\bigr\}.
\)
Hence, in the unit-capacity circulation application,
Theorem~\ref{th:dual-core} recovers a consequence implicit in the
strong-maximum-circulation theory of Atkinson et al.; its contribution
is to place this support phenomenon in the general truncated-cone
framework.

\subsection{Balanced hypergraphs}

Let $\mathcal H=(V,\mathcal E)$ be a balanced hypergraph.
Its incidence matrix is denoted by $M_{\mathcal H}$.
{Recall that } a hypergraph is balanced if its incidence matrix contains no square submatrix of odd order with all row sums and all column sums equal to $2$; equivalently, it contains no strong odd cycle in the sense of Berge \cite{berge1989}. For balanced hypergraphs, the packing polytope
$Q=\{y\in\mathbb R_+^{\mathcal E}:M_{\mathcal H} y\le \mathbf1\}$
is integral \cite{berge1989,cornuejols2001combinatorial}. Define
$K_{\mathcal H}=\cone\{\chi_F:F\in\mathcal E\}=\{x\in\mathbb R_+^V:\exists y\in\mathbb R_+^{\mathcal E},\ x= M_{\mathcal H} y\}$ and let $P_{\mathcal H}
=
K_{\mathcal H}\cap[0,1]^V$. {Since \(P_{\mathcal H}=M_{\mathcal H}Q\) and \(Q\) is the
convex hull of matchings, \(P_{\mathcal H}\) is the convex hull of
their covered-vertex vectors and is therefore integral.} We take the objective $c=\mathbf1$.
Every integral point of $P_{\mathcal H}$ corresponds to a matching of $\mathcal H$. Moreover,
$\mathbf1^\top x
=
\sum_{F\in M}|F|$,
is the number of vertices covered by the matching represented by $x$.
Thus maximizing $\mathbf1^\top x$ is the maximum covered-vertex matching
problem.

The minimal positive supports are the inclusion-minimal hyperedges. Therefore $\mathcal C$ is the family of inclusion-minimal hyperedges, and $\mathcal{D}_1 = b(\mathcal C)$ is the family of inclusion-minimal hypergraph transversals. For this cone, let $I_{\mathcal H}$ denote the set of vertices covered by every matching that maximizes the number of covered vertices.

Corollary \ref{cor:blocker} can then be translated as follows. 

\begin{corollary}
Let $\mathcal H$ be a balanced hypergraph.
The vertices covered by every matching covering a maximum number of
vertices form a transversal.
\end{corollary}

The dual cone of $K_{\mathcal H}$ is $
K_{\mathcal H}^{*}
=
\{u\in\mathbb R^{V}:M_{\mathcal H}^{\top}u\ge0\}$,
or equivalently,
$\sum_{v\in F}u_v \ge 0$
$\forall F\in\mathcal E$.

The primal optimization problem is
$
\max\{\mathbf1^{\top}x:x\in P_{\mathcal H}\}$,
whose dual is
\[
\begin{array}{ll}
\min & \mathbf1^{\top}u\\[1mm]
\mathrm{s.t.} &
\displaystyle\sum_{v\in F}u_v\ge |F|,
\qquad F\in\mathcal E,\\
&u\ge0.
\end{array}
\]

An optimal dual vector assigns nonnegative weights to the vertices so
that every hyperedge receives total weight at least its cardinality.
Since $P_{\mathcal H}$ is integral, Theorem~\ref{th:dual-core}
implies that
\(
I_{\mathcal H}
=
\bigcup
\{\supp(u):u\text{ is optimal for the dual}\}.
\)
By Remark~\ref{rem:dual-core-single-support}, there exists an
optimal dual solution $u^\circ$ satisfying
$\supp(u^\circ)=I_{\mathcal H}$.

Interval hypergraphs constitute an important subclass of balanced hypergraphs.
Their incidence matrices satisfy the consecutive-ones property and are
therefore totally unimodular.
Since the incidence matrix \(M_{\mathcal H}\) of an interval
hypergraph is totally unimodular, the cone $K_{\mathcal H}$ is box-integer and the packing polytope $
Q=\{y\geq0:M_{\mathcal H}y\leq\mathbf1\}
$
has the integer decomposition property. Moreover, every integral point
\(x\in kP_{\mathcal H}\) has an integral lift \(y\geq0\) satisfying
\(M_{\mathcal H}y=x\). Decomposing \(y\) into \(k\) integral points of
\(Q\) and projecting them through \(M_{\mathcal H}\) proves that
\(P_{\mathcal H}\) has the integer decomposition property.
Hence all assumptions of
Theorem~\ref{th:stagn}
hold.

\begin{corollary}
For interval hypergraphs,
$\mathcal D_2
=
\mathcal D_3
=
\cdots
=
\mathcal D_\infty$.
Equivalently, the inclusion-minimal supports of optimal dual solutions
are precisely the minimal level-two bottlenecks.
\end{corollary}

\subsection{Maximum-weight closures in dependency digraphs}

Let $D=(V,A)$ be a directed dependency graph. An arc $(u,v)$ means that selecting $v$ requires selecting $u$. A set $S\subseteq V$ is a closure if
$v\in S\text{ and }(u,v)\in A\quad\Longrightarrow\quad u\in S.$
Let $\mathcal F_{\rm cl}$ denote the family of closures. Define
$K_{\rm cl}=\{x\in\mathbb R_+^V:x_v-x_u\le0\text{ for every }(u,v)\in A\}$
and $P_{\rm cl}=K_{\rm cl}\cap[0,1]^V$. The integral points of $P_{\rm cl}$ are exactly the characteristic vectors of closures.

The general results apply to this cone. The coefficient matrix of the inequalities $x_v-x_u\le0$ is the transpose of a node--arc incidence matrix; augmenting it with identity rows for the box bounds preserves total unimodularity. Hence $K_{\rm cl}$ is box-integer. Moreover, $P_{\rm cl}$ has the IDP. Indeed, if $x\in kP_{\rm cl}\cap\mathbb Z^V$, define
$S_i=\{v\in V:x_v\ge i\},\qquad i=1,\ldots,k.$
Each $S_i$ is a closure: if $v\in S_i$ and $(u,v)\in A$, then $x_u\ge x_v\ge i$, so $u\in S_i$. Therefore
$x=\chi_{S_1}+\cdots+\chi_{S_k},$
which proves the IDP. Thus Theorem~\ref{th:stagn} applies.

For a utility vector $c\in\mathbb R^V$, the problem $\max\{c^\top x:x\in P_{\rm cl}\}$ is the maximum-weight closure problem \cite{picard1976}. For this cone, let $I_{\rm cl}$ denote the set of vertices belonging to every maximum-weight closure. The cone $K_{\rm cl}$ is generated by closure vectors: every nonnegative vector satisfying the dependency inequalities decomposes into its level sets. Consequently, the clutter of positive directions is
\(
\mathcal C=\min_{\subseteq}\{S\in\mathcal F_{\rm cl}:c(S)>0\}.
\)
Corollary~\ref{cor:blocker} therefore gives the following statement: the vertices belonging to every maximum-weight closure contain a transversal of the inclusion-minimal positive closures.

The higher-order theorem gives the corresponding pigeonhole form. If $S_1,\ldots,S_\ell\in\mathcal F_{\rm cl}$ satisfy
$\sum_{i=1}^\ell c(S_i)>\alpha(\ell-1),$
then
$I_{\rm cl}\cap\bigcap_{i=1}^\ell S_i\neq\emptyset.$

The dual description is also explicit. Since
$K_{\rm cl}=\cone\{\chi_S:S\in\mathcal F_{\rm cl}\},$
we have
$K_{\rm cl}^*=\{s\in\mathbb R^V:s(S)\ge0\text{ for every }S\in\mathcal F_{\rm cl}\}.$
The base dual is therefore
\begin{equation}
\min\left\{\sum_{v\in V}u_v:
\begin{array}{l}
u_v\ge0\quad(v\in V),\\
u(S)\ge c(S)\quad(S\in\mathcal F_{\rm cl})
\end{array}\right\}.
\label{eq:dual:cl}
\end{equation}
Hence, for every $\ell\ge2$,
$
J\in\mathcal B_\ell$
if and only if $J$ contains the support of an optimal solution of the dual problem \eqref{eq:dual:cl}.

Thus the bottleneck hierarchy collapses for dependency closures, and
\(
\mathcal D_2=\mathcal D_3=\cdots=\mathcal D_\infty
=\min_{\subseteq}\{\supp(u):u\text{ is optimal for \eqref{eq:dual:cl}}\}.
\)
By Theorem~\ref{th:dual-core},
\(
I_{\rm cl}
=\bigcup\{\supp(u):u\text{ is optimal for the closure-cover dual}\}.
\)

In words, a vertex belongs to every maximum-weight closure if and only if it receives positive upper-bound price in some optimal closure-cover dual solution.

\section{Concluding remarks}

We developed a conic framework for common supports of optimal solutions
in unit-box truncations. The support-transversal theorem identifies the
common optimal support with a transversal of the clutter of positive
directions, while the bottleneck hierarchy records increasingly strong
common-support obstructions. The hierarchy always stabilizes by level
\(|E|-1\).

For polyhedral cones, the dual-support sandwich gives a direct meaning
to the two ends of the hierarchy: \(\mathcal B_1\) consists exactly of
the carriers of feasible upper-bound multipliers, whereas every carrier
of an optimal multiplier belongs to \(\mathcal B_\infty\).
 {The breadth of the join-semilattice generated under union by the
zero sets of the nonzero vertices gives an upper bound on the universal
stabilization depth, and in the simplicial case this bound is exact.}
Every possible depth occurs for a  simplex, and the general bound is best possible.

Under upper-box-integrality and the integer decomposition property, the
dual and combinatorial pictures coincide from level two onward:
\(\mathfrak O=\mathcal B_2=\mathcal B_3=\cdots=\mathcal B_\infty\).
The dual-core theorem requires only integrality of the unit truncation
and identifies the common optimal support with the union of the
supports of all optimal upper-bound multipliers.

 {Several structural questions remain. The general semilattice-breadth bound}
raises the question of characterizing when it is exact and, more
generally, of finding a refinement---possibly involving the face
semilattice of \(P\)---that determines the universal stabilization
depth of a nonsimplicial truncation. A second question is to
characterize when the combinatorial equality
\(\mathcal B_2=\mathcal B_\infty\) is accompanied by the stronger
dual-carrier equality \(\mathcal B_\infty=\mathfrak O\).
It is also natural to determine when the stabilized minimal
bottlenecks cover the entire optimal core. Finally, although
\(\mathcal B_2\)-membership is co-NP-complete and fixed-parameter
tractable in \(|J|\), the complexity of recognizing or optimizing over
the inclusion-minimal families \(\mathcal D_\ell\) remains open.

\section*{Declaration of generative AI and AI-assisted technologies in the manuscript preparation process}
The authors used ChatGPT for language editing, literature-search
assistance, and critical review of the mathematical exposition.
After using ChatGPT, the authors reviewed and edited the content as needed and take full responsibility for the content of the article.

\end{document}